\documentclass[letter,reqno, 11pt]{amsart}
\usepackage{graphicx,color}
\usepackage{footmisc}
\usepackage{amsmath}
\usepackage{amsfonts}
\usepackage{amssymb}
\usepackage{bbm}
\usepackage{epstopdf}
\usepackage{color}
\numberwithin{equation}{section}
\usepackage{tikz}
\usepackage{pgfplots}
\usepackage{subfig}
\usepackage{geometry}
\newcommand{\Omm}{\Omega_{L,N;m_1,m_2}}

\newcommand{\cR}{\mathcal R}

\newcommand{\cW}{\mathcal W}
\newcommand{\bW}{\mathbb W}

\newcommand{\cM}{\mathcal M}

\newtheorem{Theorem}{Theorem}

\newtheorem{Remark}{Remark}
\newtheorem{Corollary}{Corollary}

\newtheorem{Lemma}{Lemma}

\newtheorem{Proposition}{Proposition}

 \title{Stochastic heat equation limit of a $(2+1)$d growth model}

\author{Alexei Borodin}
\address{Massachusetts Institute of Technology,
Department of Mathematics,
77 Massachusetts Avenue, Cambridge, MA 02139-4307, USA, and Institute for Information Transmission Problems, Bolshoy Karetny per. 19, Moscow 127994, Russia}
\email{borodin@math.mit.edu}

\author{Ivan Corwin}
\address{Columbia University,
Department of Mathematics,
2990 Broadway,
New York, NY 10027, USA,
and Clay Mathematics Institute, 10 Memorial Blvd. Suite 902, Providence, RI 02903, USA}
\email{ivan.corwin@gmail.com}

\author{Fabio Lucio Toninelli}
\address{Universit\'e de Lyon, CNRS and Institut Camille Jordan, Universit\'e Lyon 1, 43 bd du 11 novembre 1918, 69622 Villeurbanne, France}
\email{toninelli@math.univ-lyon1.fr}
 
\begin{document}

\begin{abstract}
  We determine a $q\to 1$ limit of the two-dimensional 
  $q$-Whittaker driven particle system on the torus studied previously in
  \cite{CT}. This has an interpretation as a $(2+1)$-dimensional stochastic interface growth model, that is believed to belong to the so-called anisotropic  Kardar-Parisi-Zhang (KPZ) class. This limit falls into a general class of
  two-dimensional systems of driven linear SDEs which have stationary
  measures on gradients. Taking the number of particles to infinity we
  demonstrate Gaussian free field type fluctuations for the stationary
  measure. Considering the temporal evolution of the stationary
  measure, we determine that along characteristics, correlations are
  asymptotically given by those of the $(2+1)$-dimensional additive
  stochastic heat equation. This confirms (for this model) the
  prediction that the non-linearity for the anisotropic KPZ
  equation in $(2+1)$-dimension is irrelevant.
 \end{abstract}
\maketitle

\section{Introduction}

The two-dimensional $q$-Whittaker particle system on the torus was
introduced by Corwin and Toninelli \cite{CT}. The state of this system
is interlacing collections of particles on the two-dimensional
torus. Particles jump right by one on their row according to
exponential clocks whose rates are determined by certain nearest
neighbor inter-particle distances -- see (\ref{eq:6prime}) -- as well
as a parameter $q\in[0,1)$.  As discussed in \cite{CT}, the dynamics
can be  seen also as a stochastic growth process for a random discrete
$(2+1)$-dimensional interface. The mapping is based on the fact that
interlacing particle configurations correspond to perfect matchings
(dimer coverings) of the hexagonal lattice, and the associated
integer-valued height function provides the discrete interface. In the
present work, we adopt the point of view of the particle system rather
than that of the growth model.

These dynamics originated in Borodin and Corwin's study of Macdonald
processes \cite{BorCor} wherein they were defined on certain
triangular arrays of interlacing particles. Based on that inspiration
as well as a recent treatment by Toninelli in \cite{Ton} of the $q=0$
case of this system, \cite{CT} proposed and then verified that certain
local Gibbs measures are stationary for the dynamics -- see
(\ref{eq:3prime}) and Proposition \ref{Prop1} for a summary of those
results.

When $q=0$, 
\cite{Ton} determined that stationary measures are simply 
uniform measures on particle configurations, which are known to enjoy a relation to determinantal point
processes and to have Gaussian Free Field type fluctuations in
the infinite volume limit for the torus. In \cite{Ton} it was also shown
that the dynamics remain well-defined in the infinite volume limit and
bounds on the scale of fluctuations of the
associated height function for the system were determined. In
particular, it was shown that height
function fluctuations grow at a smaller rate than any polynomial in
time $t$ (and that the bound can be improved to $O(\sqrt {\log t})$ in a certain range of particle densities).

The $q=0$ model is predicted (cf. \cite{BorFer}) to be in the anisotropic $(2+1)$-dimensional Kardar-Parisi-Zhang universality class which is represented by the stochastic PDE
\begin{equation}\label{AKPZ}
\frac{\partial h}{\partial t}(t;x,y) = \frac{1}{2}\Delta h(t;x,y) + (\nabla h, Q \nabla h)(t;x,y) + \xi(t;x,y)
\end{equation}
where $h(t;x,y)$ represents a height function at time $t$ and position
$(x,y)$, $\Delta$ is the Laplacian in $x$ and $y$, $Q$ is a $2\times
2$ matrix with signature $(+,-)$ and $\xi(t;x,y)$ is a space-time
white noise. This is called anisotropic because of the mixed signature
on the non-linearity, whereas when the signature of $Q$ is $(+,+)$ or
$(-,-)$ the model is called ``isotropic''. In 1991, Wolf \cite{Wolf} predicted that the fluctuations of the anisotropic equation should grow like $\sqrt{\log t}$ and moreover that the non-linearity should be irrelevant and the long-time behavior of the system should be exactly as that of the two-dimensional additive stochastic heat equation (i.e., the equation with $Q$ set to zero).

Wolf's prediction of $\sqrt{\log t}$ fluctuations was demonstrated
numerically soon after by Halpin-Healy and Assdah
\cite{HalpinHealyAssdah}. Pr\"{a}hofer and Spohn \cite{PS} considered
a microscopic model related to the Gates-Westcott model, and
demonstrated through exact calculation this fluctuation
scaling. Borodin and Ferrari \cite{BorFer} studied a (non-periodic)
triangular array variant of the $q=0$ case of the $q$-Whittaker
particle system and, using the technology of Schur processes they
proved the $\sqrt{\log t}$ scaling and further demonstrated
convergence to a Gaussian free field as time goes to infinity. The
occurrence of a Gaussian free field is consistent with the prediction
of convergence to the additive stochastic heat equation since the
Gaussian free field is stationary for the time evolution of that
stochastic PDE. It should be noted that since \cite{BorFer} dealt with
dynamics on a triangular array of particles, the Gaussian Free Field
fluctuations only appear after a suitable coordinate change. This
coordinate change is not visible from the renormalization group
arguments of Wolff. In this paper we focus on translation invariant
models for which no coordinate change is necessary.

We should note that (\ref{AKPZ}) is not, a priori, well-defined
because of the non-linearity and the fact that solutions are not
function valued, but rather generalized functions (like the Gaussian
free field). We are not aware of any rigorous treatment of this
equation, though it may eventually fall into the class of systems
which can be defined through Hairer's regularity structures
\cite{HairerRegularity}. As such, none of Wolf's predictions have been
proved for (\ref{AKPZ}) itself. Let us also note that the story is
quite different when the model is isotropic -- see recent numerical
studied of Halpin-Healy \cite{HHisotropic}.

To our knowledge, the present paper, together with the
  forthcoming work of Borodin, Corwin and Ferrari \cite{BCF}  in the
  context of these dynamics on triangular arrays, is the first
instance in which the full space-time picture has been rigorously
established for a model in the $(2+1)$-dimensional anisotropic KPZ
class converging to the additive stochastic heat equation. To be
upfront about things, we do not prove this convergence as a process  (which would require some additional tightness estimates),
but rather in terms of the correlation structure for Gaussian
processes.

The initial motivation for this paper was the desire to extend the
study of \cite{BorFer,Ton} to the $q\neq 0$ case. In that case the
Schur / determinantal point process structure is lost. To overcome
this impediment, we decided to consider a Gaussian limit of the model,
hoping that calculations there would become sufficiently doable
without said structure. In particular, we consider the
$q=e^{-\epsilon}\to 1$ limit of the particle system, as we
simultaneously scale the torus width and height like
$\epsilon^{-1}$. We start  particles spaced on the $\epsilon^{-1}$
scale according to a certain crystalline configuration (see the
beginning of Section \ref{sec:conv}) with smaller $\epsilon^{-1/2}$
scale fluctuations.  Speeding time up by $\epsilon^{-1}$, we prove
(Theorem \ref{th:SDE}) that particle positions (multiplied by $\epsilon$)
have asymptotically a constant
 speed $v$ and
that fluctuations (multiplied by $\epsilon^{1/2}$)
 converge (as a space-time process) to a
limiting system of SDEs. Likewise, under this scaling the stationary
measure on the $q$-Whittaker system converges to a Gaussian measure
(Lemma \ref{th:locmax}). (Note: we do not prove that the stationary
measure concentrates on the crystalline configuration, though it is
certainly compelling to conjecture this).

Once in the setting of SDEs with Gaussian stationary measures (in
fact, the stationary measures are on gradients) we are able to use
Fourier transforms to explicitly compute the space-time correlations
as the number of particles goes to infinity (Theorem \ref{th:2}) as
well as the correlations and Gaussian free field limit of the
stationary measure (Theorem \ref{th:muinv}). Theorem \ref{th:2} has a
number of corollaries. Corollary \ref{cor1} shows that for fixed,
large time $t$, fluctuations grow like $\sqrt{\log t}$ and
correlations decay in a spatial range of order $t^{1/2}$. Corollary
\ref{cor2} considers the correlations along space-time lines. There
exists a distinguished direction $U$ along which correlations exist in
a temporal scale of order $t$ and a spatial scale of order
$t^{1/2}$. In fact, Corollary \ref{cor3} shows that in this scale, the
correlations converge to those of the $(2+1)$-dimensional additive
stochastic heat equation -- thus validating Wolf's prediction for this
model. On the other hand, Corollary \ref{cor2} also shows that for
space-time direction not equal to $U$, the correlations decay to zero on a $t^{1/2}$ time-scale, thus much faster than along $U$.

 We call the direction $U$ the
``characteristic'' direction. The reason is that the lines $y=U t$ are the characteristic curves of a PDE that we conjecture to describe the hydrodynamic limit of our model when the initial condition is not crystalline, see Section \ref{sec:pde}. There is a close
analogy with what happens in the context of $(1+1)$-dimensional
particle systems.  In that setting, the hydrodynamics are described by
Hamilton-Jacobi conservation laws which can be solved by computing the
characteristics and propagating initial data along
them. Characteristics are computed as the derivative of the flux with
respect to the local slope. At a more microscopic level, initial data
fluctuations are propagated along characteristics. In particular, one
has ``slow decorrelation'' \cite{Ferrari,CFPslowdec} along characteristics
whereby fluctuations along these space-time directions decorrelate far
slower than along other space-time directions. The phenomenon of slow decorrelation along characteristics was conjectured (with some supporting evidence) in \cite{BorFer} to hold
for the two-dimensional $q$-Whittaker particle system at $q=0$.

The aforementioned results concerning correlations of two-dimensional
systems of SDEs are actually proved below in much broader generality. In
particular, all results are proved provided that the SDEs take the
form of \eqref{eq:2} with the matrix $A$ satisfying the conditions of
Proposition \ref{prop:prop}. This  could reflect the expected   universality of
the $(2+1)$-dimensional anisotropic KPZ class.
It would be compelling
to see if any of these universality results can be proved directly for
the general $q$ system, without first taking the SDE limit. Without
the Gaussian structure, though, it is unclear how to proceed in this
goal.

As mentioned above, in the $q=0$ case, there exists a triangular array
variant of the $q$-Whittaker particle system which was studied at
length in \cite{BorFer} using Schur processes. The triangular variant
of the general $q$ case relates to $q$-Whittaker processes
\cite{BorCor} and though the system is no longer determinantal, there
are many useful formulas provided through the technology of Macdonald
processes. In a parallel paper to this, Borodin, Corwin and Ferrari
\cite{BCF} develop the analogous $q\to 1$ limit of this triangular
variant of the particle system and explore the limits of the exact
formulas and their applications (in particular, concerning the
asymptotic behavior of correlations along certain space-time
directions).

\subsection{Acknowledgements}

The authors wish to thank Patrik Ferrari for conversations
on   this work and the related work \cite{BCF}.  A. B. was
partially supported by the NSF grant DMS-1056390. I. C. was partially
supported by the NSF DMS-1208998, by a Clay Research Fellowship, by
the Poincar\'{e} Chair, and by a Packard Fellowship for Science and
Engineering. F. T. was partially funded by Marie Curie IEF Action
“DMCP- Dimers, Markov chains and Critical Phenomena”, grant agreement
n. 621894. This work was initiated during the Statistical Mechanics,
Integrability and Combinatorics program at Galileo Galilei Institute (Arcetri). We appreciate the
hospitality and support of these institutes.

\section{Model and notation}

We start by recalling the definition of the $q$-Whittaker particle
system on the torus, and its stationary measure as defined in
\cite{CT}. We consider an interacting particle system in which
particles live on the $L\times N$ discrete torus $\mathbb
T_{L,N}=\mathbb Z/(L \mathbb Z)\times \mathbb Z/(N \mathbb Z)$. The
horizontal size is $L$ and the vertical size is $N$.

The particle configuration space will be denoted $\Omm$, and it depends
on two integers $1< m_1< L$ and $1\le m_2< N$ such that
\begin{gather}
  \label{eq:26}
m_1/L+m_2/N<1.
\end{gather}
At each site $x=(x_1,x_2)\in \mathbb T_{L,N}$ there is at most one particle. On each row there are exactly $m_1$ particles. We exclude $m_1=1$ and $m_1=L$ to avoid trivialities. The parameter $m_2$ has a more topological nature and its meaning will be explained below.

The horizontal position of particle $p$ is denoted  $x_p\in \mathbb Z/(L\mathbb Z)$. Particle positions are interlaced, in the following sense.  Given  particle $p$ (say on row $i$), we let $p_1,p_4$ denote its right/left neighbor on the same row (note that if $m_1=2$ then $p_1=p_4$). Then, we require that in row $i-1$ there is exactly one particle, labeled $p_2$, whose position satisfies
\begin{eqnarray}\label{eq:1prime}
x_{p_2}\in \{x_{p}+1,x_p+2,\dots,x_{p_1}\}
\end{eqnarray}
and exactly one particle, labeled $p_3$, satisfying
\begin{gather}
  \label{eq:22}
  x_{p_3}\in \{x_{p_4}+1,x_{p_4}+2,\dots,x_{p}\}.
\end{gather}
See Figure \ref{fig:particelle}. Note that, automatically, in row $i+1$ there are exactly one particle $p_5$ and one
particle $p_6$ satisfying respectively
\begin{gather}
  \label{eq:23}
  x_{p_5}\in \{x_{p_4},\dots,x_p-1\},\quad
x_{p_6}\in \{x_{p},\dots,x_{p_1}-1\}.
\end{gather}

\begin{figure}[ht]
\includegraphics[width=.6\textwidth]{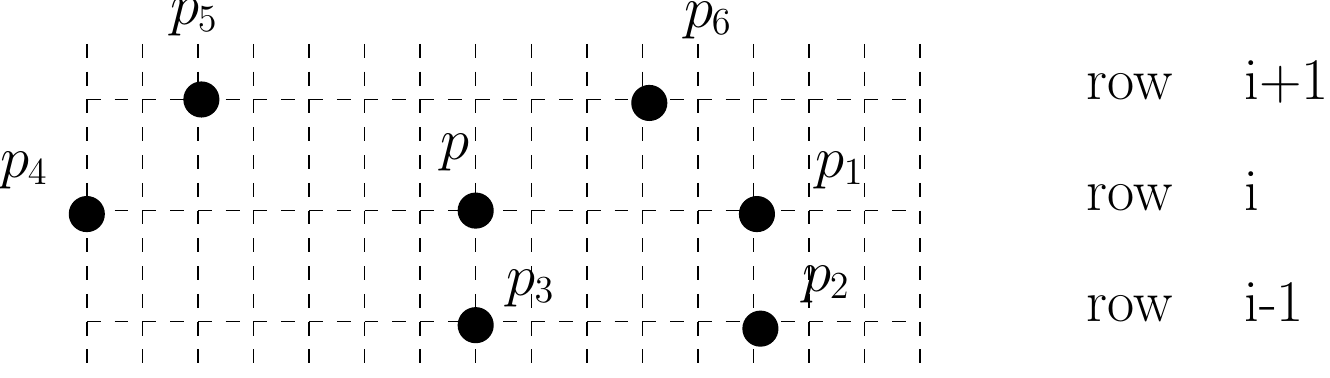}
\caption{
The neighbors $p_1,\dots,p_6$ of particle labeled $p$. Note that conditions
\eqref{eq:1prime}, \eqref{eq:22} allow $C_p:=x_{p}-x_{p_3}=0$ but they impose
$B_p+1:=x_{p_2}- x_p\ge1$.
}
\label{fig:particelle}
\end{figure}
We define non-negative integers $A_p,\dots,F_p$ as
\begin{eqnarray}
  \label{eq:19}
  &A_p=x_{p_1}-x_{p}-1;\quad
B_p=x_{p_2}-x_p-1;\quad
C_p=x_p-x_{p_3}\\
\nonumber
&D_p=x_{p}-x_{p_4}-1;\quad
E_p=x_p-x_{p_5}-1;\quad
F_p=x_{p_6}-x_p.
\end{eqnarray}
The particles $p_1,\dots,p_6$ are the six neighbors of $p$, labeled clockwise starting from the one on the right. The definition of the dynamics will be such that the labels of the neighbors of a particle $p$ do not change with time
(particles will not jump over each other or change interlacements).

Let $\Omega_{L,N;m_1}$ be the set of particle occupation functions,
i.e. of functions $\eta:\mathbb T_{L,N}\mapsto \{0,1\}$, with $m_1$
particles (i.e. occupation variables equal to 1) per row, whose
positions satisfy the constraints \eqref{eq:1prime}-\eqref{eq:23}. The set
$\Omega_{L,N;m_1}$ decomposes into disjoint ``sectors'':
\begin{gather}
  \label{eq:24prime}
  \Omega_{L,N;m_1}=\bigcup_{m_2}\Omm
\end{gather}
as follows. Given any particle $p$, connect $p$ to its up-right
neighbor $p_6$, then $p_6$ with its own up-right neighbor and repeat
the operation until the path $\Gamma$ thus obtained gets back to the
starting particle $p$. Note that $\Gamma$ forms a simple loop:
otherwise, there would be a particle $r$ which is reached along
$\Gamma$ from two different particles $r',r''$. This is impossible,
since both $r'$ and $r''$ would be the lower-left neighbor $r_3$ of
$r$. Call $N_v\in\mathbb N\cup\{0\},N_h\in\mathbb N\cup\{0\}$ the
vertical and horizontal winding numbers of $\Gamma$ around the torus
$\mathbb T_{L,N}$. It is easy to see that $N_h,N_v$ are independent of
the chosen initial particle $p$.  As discussed in \cite[Remark 2]{CT},
\begin{gather}
  \label{eq:27}
  m_2:=m_1\frac{N_h}{N_v}
\end{gather}
is an integer and it satisfies \eqref{eq:26}. The set
$\Omega_{L,N;m_1,k}$ is defined as the subset of $\Omega_{L,N;m_1}$ with $
m_2=k$. Each sector $\Omm$ will remain invariant under our
dynamics.

Let us briefly remark that the particle configurations we
are considering can also be mapped onto dimers on the periodized
$L\times N$ 
hexagonal lattice. This perspective is explained in \cite{CT} (see in particular Fig. 2 there) wherein
$n_1:=m_1 N$ corresponds to the number of vertical dimers, $n_2:=m_2 L$ to the
number of north-west dimers and $n_3:=NL-n_1-n_2$ the number of
north-east dimers. We will not pursue this perspective any further
here.

Given a configuration $\eta\in \Omm$, draw a directed upward edge from
any particle label $r$ to its up-right neighbor $r_6$ if $F_r=0$ (in
which case $r$ and $r_6$ have the same horizontal
position). For the particle labeled by $p$ let $V^+_p$ be the set that
includes $p$ plus the particle labels that can be reached from $p$ by
following upward oriented edges. The dynamics we consider is a
continuous-time Markov chain on $\Omm$.
For each $p$, there is an exponential clock of rate 
\begin{eqnarray}
  \label{eq:6prime}
\frac{(1-q^{B_p})(1-q^{D_p+1})}{(1-q^{C_p+1})}
\end{eqnarray}
 with $q\in[0,1)$. When said clock rings, all particles with label $r\in V^{+}_p$ shift by $(1,0)$. 
 Note that the rate is zero if $B_p=0$. This
prevents particles from overlapping after the move. Note also that
after the move, the configuration is still in $\Omm$. This is
discussed in more detail in \cite{CT}. Another way to
understand the dynamics is that when particle $p$ moves, provided its
up-right neighbor $p_6\in V^+_p$, then $C_{p_6}$ becomes $-1$ and the
jump rate for $p_6$ becomes infinite, and hence it immediately moves
(and so on for all other $r\in V^+_p$). These dynamics are called the
$q$-Whittaker particle system on the torus.  As a side remark, let us add that, in terms of dimer model, shift to the right by $+1$ of a family $V^+_p$ corresponds to increasing the height by $1$ in $|V^+_p|$ faces of the hexagonal lattice.

Given $q\in [0,1)$, let $\pi$ be the probability measure on
$\Omm$ defined as
\begin{eqnarray}
  \label{eq:3prime}
  \pi(\sigma):=\frac1{Z_{L,N;m_1,m_2}}\prod_{p}
 \frac{(q;q)_{A_p}}{(q;q)_{B_p}(q;q)_{C_p}}{\bf 1}_{\{\sigma\in \Omm\}}
\end{eqnarray}
where $(q;q)_n=(1-q)(1-q^2)\dots (1-q^n)$ 
and $Z_{L,N;m_1,m_2}$ is the normalizing constant necessary to make this a probability measure.
The main result of \cite{CT} (Theorem 1 therein) is:
\begin{Proposition}
\label{Prop1}
The probability law $\pi$ is stationary in time for $q$-Whittaker particle system on the torus.
\end{Proposition}

Clearly, the measure $\pi$ is not reversible, since the process is totally asymmetric.

\bigskip

We will consider a certain $q\to 1$ limit of the particle system. To fix scalings, for $\epsilon>0$ let $q=\exp(-\epsilon)$. For simplicity (to avoid a plethora of $\lfloor \cdot \rfloor$), we will assume $\epsilon^{-1}$ is an integer, though all results hold for arbitrary $\epsilon$.
 Further, let $L=\epsilon^{-1}\ell$ for some integer $\ell$. On each row there are $m:=m_1=n_1/N$
particles, with $m$ an integer of order $\ell$.  Also for simplicity we take $N=m$, so altogether
we have $n_1=m^2=O(\ell^2)$ particles. Taking $N$ to be another
multiple of $m$ would not change our results qualitatively.

  \begin{Remark}
\label{rem:bcc1}
On $\mathbb Z^2$, introduce the equivalence relation $\sim$ such that
$p\sim p'$ iff $p=p'+(j_1m-j_2 m_2,j_2 m)$ for some $j_1,j_2\in\mathbb Z$.
Observe that if we take $N=m$ steps along $\Gamma$ starting from $p$, we get a particle $p'$ that is the $j^{th}$ right neighbor of $p$ on the same row, for some $0\le j\le m$. Actually, one has $j=m_2$, since $N_vj=N_h m$.
Particles will be given a label $p\in \mathbb Z^2$ such that 
 $p_1=p+(1,0), p_2=p+(1,-1), p_3=p+(0,-1),p_4=p-(1,0),p_5=p+(-1,1),p_6=p+(0,1)$,
 with an arbitrary choice of which particle is labeled $p=(0,0)$. Thanks to the above observation, given integers $j_1,j_2$, particle $p+(j_1 m_1,j_2 N)$ is to be identified with $p+(j_2 m_2,0)$. 
In other words, particles are labeled by
$p\in \cR_m:=\mathbb Z^2/\sim$ (the quotient set of $\mathbb Z^2$ by
$\sim$, which contains $m^2$ equivalence classes).    
  \end{Remark}

We will first take the limit $\epsilon\to0$ with $m$ and $m_2$ fixed, and then
 $m\to\infty$ with the ratio
 $m_2/m$  bounded away from 0 and 1 if we want to take an
infinite-volume limit ($m_2/m\le1 $  by definition, recall
\eqref{eq:26} and $m=N$).

The average inter-particle distance along
a row is
\begin{eqnarray}
  \label{eq:D}
\frac{D}{\epsilon}:=\frac Lm=\frac{\ell}{\epsilon m}
\end{eqnarray}
so that, with reference to \eqref{eq:19}, the average value of  $D_p$ is $\epsilon^{-1}D$.

The average value of $C_p$ is  instead  fixed by choosing $m_2$ suitably, as follows. Let
 $C\in (0,D)$ be defined by
\begin{eqnarray}
  \label{eq:C}
\frac CD  =\frac{m_2}m
\end{eqnarray}
and let $B=D-C$.
(Note that, as long as the integers $m,m_2$ are  kept finite, $C$ can take only a finite set of values in $(0,D)$; such set becomes dense in $(0,D)$ if we allow $m,m_2$ to diverge with a fixed ratio).

We claim that  the average value of $C_p$ in any particle configuration  is $\epsilon^{-1}C$. Indeed, take the
path $\Gamma$ defined after \eqref{eq:24prime}. It visits exactly $N_v\times N=m N_v$ particles and its total displacement to the right equals $N_h \times L$. On the other hand, $N_h \times L$ is also the sum of the $C_p$ over all particles in $\Gamma$. 
From this, we see that the average value of $C_p$ is
\begin{eqnarray}
  \label{eq:bo}
\frac{N_h L}{N_v m}=\frac{m_2 L}{m^2}= \epsilon^{-1}\frac {m_2}{m }\frac{\epsilon L}{m}=\epsilon^{-1}\frac {m_2}{m}D
\end{eqnarray}
where in the first equality we used \eqref{eq:27}, and in the third \eqref{eq:D}. Eq. \eqref{eq:C} then allows to conclude.

Similarly, we see that the average of $B_p$ is  $\epsilon^{-1}B$. On the other hand, the averages of $D_p,E_p,F_p$ equal by definition the averages of $A_p, B_p,C_p$ respectively.

\section{Convergence to a system of SDEs}\label{sec:conv}

We will start the dynamics from an initial condition where each particle $p\in \cR_m$ is within distance $O(\epsilon^{-1/2})$
from its ``ideal position'' $X_p$ in a perfect ``crystalline
configuration'' where $D_p=\epsilon^{-1}D,
B_p=\epsilon^{-1}B,C_p=\epsilon^{-1}C$ for every $p$.
Assuming without loss of generality that $X_0=0$, we have for $p=(p^{(1)},p^{(2)})$
\begin{eqnarray}
  \label{eq:Xp}
  X_p=p^{(1)} \epsilon^{-1}D+p^{(2)}\epsilon^{-1}C
\end{eqnarray}
where the r.h.s. has to be taken modulo $\epsilon^{-1}\ell$.

Our first result (proven in Section \ref{sec:t1}) says that particles move macroscopically with a deterministic speed $v(C,D)>0$ and that, in the time-scale of order $\epsilon^{-1}$, fluctuations around such hydrodynamic limit are of order $\epsilon^{-1/2}$ and
converge   to a system of linear SDEs.

\begin{Theorem}
\label{th:SDE} Fix $N=m$ and $\ell$, so that the lattice $\Lambda_{L,N}$ depends only on $\epsilon$.
 Let $\{\bar\xi_p\}_{p\in \cR_m}\in \mathbb R^{\cR_m}$. Consider an
 initial configuration  $\sigma_0$ such that, defining 
 \begin{eqnarray}
   \label{eq:etaas}
   \eta_p:=\eta_{p,\epsilon}=\sqrt{\epsilon}(x_p-X_p),
 \end{eqnarray}
one has
 \begin{eqnarray}
   \label{eq:xi}
   \lim_{\epsilon\to0}\eta_p=\bar\xi_p \quad \forall p\in  \cR_m.
 \end{eqnarray}
Define
\begin{gather}
  \label{eq:1}
  v(C,D)=\frac{(1-\exp(-B))(1-\exp(-D))}{1-\exp(-C)}
\end{gather}
(recall that $B=D-C$) and
\begin{eqnarray}
  \label{eq:xit}
  \eta_{p,t}=\sqrt\epsilon\left({x_p(t/\epsilon)-X_p-\epsilon^{-1}vt}\right),
  \quad t\ge0.
\end{eqnarray}
Then, the random process $\{\eta_{p,\cdot}\}_{p\in\cR_m}$
converges weakly as $\epsilon\to0$ to the solution of the system of
linear stochastic differential equations
\begin{gather}
  \label{eq:2}
\left\{
  \begin{array}{ll}
 d\xi_{p,t}=\sqrt v\, {\rm  d}W_{p,t}+\sum_{p'}A_{p,p'}\xi_{p',t} dt\\
\xi_{p,0}=\bar\xi_p
  \end{array}
 \right.
\end{gather}
with
\begin{gather}
\label{eq:A}
A_{p,p'}=\delta_{p'=p}\left(\frac{e^{-D}(1-e^{-B})}{1-e^{-C}}-\frac{e^{-B}(1-e^{-D})}{1-e^{-C}}-\frac{e^{-C}(1-e^{-B})(1-e^{-D})}{(1-e^{-C})^2}\right)\\
+\delta_{p'=p+(1,-1)}\frac{e^{-B}(1-e^{-D})}{1-e^{-C}}
-\delta_{p'=p-(1,0)}\frac{e^{-D}(1-e^{-B})}{1-e^{-C}}\\+
\delta_{p'=p-(0,1)}\frac{e^{-C}(1-e^{-B})(1-e^{-D})}{(1-e^{-C})^2}
\end{gather}
and where $dW_{p,t}$ are independent white noises indexed by $p$, and one-dimensional in time $t$.
\end{Theorem}

\begin{Remark}
Note  that the matrix $A$ is not symmetric, so that
the diffusion \eqref{eq:2} is not reversible with respect to its
stationary  measure described by Theorem \ref{th:muinv}. 
In other words, the irreversibility of the microscopic dynamics survives also in the Gaussian
  limit.
\end{Remark}

\subsection{Properties of the matrix $A$}
It is convenient to work in Fourier space. For this purpose,
let
\begin{gather}
  \label{eq:6}
  f_k:\cR_m\mapsto \mathbb C, \quad f_k: p\mapsto \frac1m e^{-i p k}.
\end{gather}
  The set
$\{f_k\}_{ k\in \mathcal K_{m}}$,
 where
\begin{gather}
\mathcal K_{m}=\{((2\pi/m)r_1,(2\pi/m)\left(\frac CD r_1+r_2\right), r_1,r_2\in\mathbb Z, -m/2\le r_1,r_2<m/2\},
\end{gather}
forms an
orthonormal basis of $\mathbb C^{\cR_m}$. 
\begin{Remark}
Note that $\mathcal K_m$ was chosen such that, if we define
$f_k(p)=(1/m)e^{-i p k}$
for every $p\in\mathbb Z^2$, then $f_k(p)=f_k(p')$ if
$p\sim p'$ (use that $m_2/m=C/D$, see \eqref{eq:C}).
Also, observe that $|\mathcal K_m|=m^2$.
\end{Remark}
Define
\begin{eqnarray}
  \label{eq:hatxi}
  \hat \xi_{k,t}=\sum_{p\in \cR_m} \xi_{p,t}f_k(p),
\end{eqnarray}
so that
\begin{eqnarray}
  \label{eq:eta-eta}
 \xi_{p,t}=\sum_{k
\in \mathcal K_{m}}\hat\xi_{k,t}\overline{ f_k(p)},\quad \text{ and}\quad  \hat\xi_{-k,t}=\overline{\hat \xi_{k,t}}.
\end{eqnarray}
Let also
\begin{eqnarray}
  \label{eq:Ak}
  \hat A(k)=\sum_p A_{p,0}e^{-i p k}, k\in\mathbb R^2.
\end{eqnarray}

In our specific case, one sees that
\begin{eqnarray}
  \label{eq:9}
  \hat
  A(k)&=&A_{0,0}+A_{0,(1,-1)}e^{-i(k_1-k_2)}+A_{0,(-1,0)}e^{ik_1}+A_{0,(0,-1)}e^{i
    k_2}\\
\hat  A(k)+\hat
 A(-k)&=&2[A_{0,0}+A_{0,(1,-1)}\cos(k_1-k_2)\\&+&A_{0,(-1,0)}\cos(k_1)+A_{0,(0,-1)}\cos(k_2)].
\end{eqnarray}

Observe that we defined $\hat A(k)$ for any $k\in \mathbb R^2$ and not
just for $k\in \mathcal K_{m}$.

\begin{Proposition}
\label{prop:prop}
  The matrix $A$ satisfies the following properties:
  \begin{enumerate}
  \item Translation invariance: $A_{p,p'}=A_{p+r,p'+r}$ for every $r$;
  \item $\hat A(0)=\sum_{p'} A_{p,p'}=0$;
  \item 
\begin{eqnarray}
  \label{eq:37}
\hat R(k):=\hat A(k)+\hat A(-k)\le 0
\end{eqnarray}
and the only zero of $\hat R(k)$ on $[-\pi,\pi]^2$ is at $k=0$;
\item 
  \begin{eqnarray}
    \label{eq:rkw}
\hat R(k)= \hat W(k)+O(|k|^3)=(k, \hat W k)+O(|k|^3),    \quad k\to0,
  \end{eqnarray}
 with
$\hat W$ a strictly negative definite $2\times 2$ matrix and $\hat W(k)$ the associated quadratic form.
\item The function  $\hat A(\cdot)$ is $2\pi$-periodic and $C^\infty$ on $[-\pi,\pi]^2$.
  \end{enumerate}
\end{Proposition}

All properties are trivial to check, except for (3) which is proven in
Appendix \ref{app:Rk}. The fact that
\begin{gather}
\hat W(k):=-A_{0,(1,-1)}(k_1-k_2)^2-A_{0,(-1,0)}k_1^2-A_{0,(0,-1)}k_2^2
\end{gather}
is a strictly negative definite quadratic form
follows from negativity of $\hat
R(k)$, together with the fact that
\begin{eqnarray}
  \label{eq:38}
\det(\hat W)=\frac{e^{-D}(1-e^{-D})(1-e^{-B})^2}{(1-e^{-C})^2}=:w^2>0.
\end{eqnarray}
For later convenience, let  the $2\times 2$ matrix $V$ be such that
\begin{eqnarray}
  \label{eq:Vv}
V \hat W V^T=-I.
\end{eqnarray}

Let us also define
  \begin{eqnarray}
    \label{eq:Uu}
    U:=i\nabla \hat
    A(0)=\sum_{p}p A_{0,p}=  (A_{0,(1,-1)}-A_{0,(-1,0)},-A_{0,(1,-1)}-A_{0,(0,-1)})\in
    \mathbb R^2,
  \end{eqnarray}
 and note that $U\ne 0$ for any choice of $B,C,D=B+C$.

\begin{Remark}
A few comments are in order:
  \begin{itemize}
  \item Property (1) (translation invariance of $A$) originates from
    the fact that we are considering fluctuations around a
    ``trivial'', translation invariant hydrodynamic limit where
    particles have positions $X_p+\epsilon^{-1}v t$ and are therefore
    equi-spaced at all times. This is the property that makes it
    convenient to work in Fourier space (e.g., property (1) is behind
    the fact that Eq. \eqref{eq:dxik} is diagonal in $k$);

  \item Property (2) means that the  drift of a particle $p$ is unchanged if all
    particles are globally shifted by the same amount. This is a consequence of the fact that transition rates  \eqref{eq:6prime} of the microscopic
    particle process  depend only on inter-particle
    distances, and would hold even if we studied fluctuations w.r.t. a
    non-translation invariant  hydrodynamic limit (cf. Section \ref{sec:pde});

  \item Property (3) guarantees that there exists a stationary measure for
gradients $\xi_p-\xi_{p'}$, see Theorem \ref{th:muinv} and formula
\eqref{eq:inv}.
Negative-definiteness of $\hat A(k)+\hat A(-k)$ corresponds to the
fact that
in the hydrodynamic scaling the particle configuration is ``crystalline'' at all
times, and that crystalline  configurations are an (at least local)
maximum of the stationary  measure $\pi(\cdot)$,
see Lemma \ref{th:locmax}.

  \end{itemize}

\end{Remark}

\subsection{A conjectural hydrodynamic equation}

\label{sec:pde}
We will briefly consider the hydrodynamic behavior of the
$q$-Whittaker particle system with general initial data (which may not
be close to the crystalline configurations considered above). We
provide a heuristic derivation of the hydrodynamic (law of large
numbers) PDE satisfied by the limit. For us, the purpose of this
derivation is to justify and explain our use of the word
"characteristic" to describe the direction $U$ of slow
decorrelation. In this hydrodynamic limit we will first take
$\epsilon\to0$ and then $m\to\infty$. With this in mind, let us rescale
space and time by $m/\epsilon$ and define 
\[
h(\tau,y)=(\epsilon/m) x_{\lfloor y m\rfloor }(\tau m/\epsilon)
\]
with $y \in [0,1]^2$, $\tau\ge0$, and
where the r.h.s. has to be taken modulo
$(\ell/\epsilon)\times(\epsilon/m)=D$ because $x_p$ is defined modulo $L=\ell/\epsilon$.
Note that, as a consequence of Remark \ref{rem:bcc1}, $h(\tau,\cdot)$ satisfies
\[
h(\tau,y+(j_1,j_2))=h(\tau,y+(Cj_2/D,0)), \quad j_1,j_2\in \mathbb Z.
\]
If at time zero the configuration satisfies the
conditions of Theorem \ref{th:SDE}, then $h(0,y)$ is close to a  
linear function:
\[
H(0,y):=\lim_{m\to\infty}\lim_{\epsilon\to0}h(0,y)=D y_1+C y_2.
\]
In this case, it follows from Theorem \ref{th:SDE} that, for $\tau>0$,
the limit 
\begin{eqnarray}
  \label{eq:hlim}
H(\tau,y):=\lim_{m\to\infty}\lim_{\epsilon\to0}h(\tau,y)  
\end{eqnarray}
exists and solves
\begin{eqnarray}
  \label{eq:55}
\partial_\tau H=v   
\end{eqnarray}
with $v(C,D)$ as in \eqref{eq:1}.
Given that $D=\partial_{y_1} H$ and $C=\partial_{y_2} H$, and using
$B=D-C$, we see that  
\begin{eqnarray}
  \label{eq:56}
v(C,D)=v(\nabla H)=\frac{(1-\exp(\partial_{y_2}H-\partial_{y_1}
  H))(1-\exp(-\partial_{y_1}H))}{(1-\exp(-\partial_{y_2}H))}.
\end{eqnarray}

If we assume instead that the initial condition satisfies
\[
\lim_{m\to\infty}\lim_{\epsilon\to0} h(0,y)=H(0,y),
\]
with $H(0,\cdot)$ some smooth enough but non-linear function, then
it is natural to conjecture that the limit \eqref{eq:hlim}
exists and still satisfies \eqref{eq:55}, with $v=v(\nabla H)$ equal
to the r.h.s. of \eqref{eq:56}.
It is easy to see that the characteristic lines of the PDE
\eqref{eq:55} are the straight lines
$y=U t$ with $U$ as in \eqref{eq:Uu}. In fact, the characteristic
lines are determined by 
\[
y_i(t)=t \nabla_i v,\quad i=1,2
\]
with $\nabla_i v$ the derivative of $v=v(\nabla h)$ with respect to
its $i$-th argument. Explicitly,
\[\nabla_1 v=\frac{e^{\partial_{y_2}H-\partial_{y_1}H}(1-e^{-\partial_{y_1}H})}{1-e^{-\partial_{y_2}H}}+\frac{e^{-\partial_{y_1}H}(1-e^{\partial_{y_2}H-\partial_{y_1}H})}{1-e^{-\partial_{y_2}H}}=A_{0,(1,-1)}-A_{0,(-1,0)}=U_1
\]
and similarly $\nabla_2 v=U_2$. We emphasize that the identity
$\partial_i v=U_i$ is not
a coincidence. Indeed, view the hydrodynamic speed $v$ as a function of the relative
horizontal distances
between particle, say, $0$ and the other particles $p$ (in our case, the
dependence is only through the three neighbors $p_2,p_3,p_4$ of $0$).
On one hand, since the SDEs \eqref{eq:2} describe a linearization of the
stochastic 
dynamics around the hydrodynamic limit, $A_{0,p}$ is obtained taking the
derivative of $v$ w.r.t. the position $x_p$ of particle $p$, with the
others kept fixed. On the other hand, if the slope $\nabla_i H$ is
changed by $\epsilon$, the distance between particles $0$ and
$p=(p^{(1)},p^{(2)})$ changes by $\epsilon\times p^{(i)}$.
Therefore,
\begin{eqnarray}
  \label{eq:16}
  \nabla_{i} v=\sum_{p=(p^{(1)},p^{(2)})}p^{(i)}A_{0,p}=U_i.
\end{eqnarray}

This shows that the identity  $U_i =\nabla_i v$ is not related to the particular form of the function
$v(\cdot)$. If we had another interacting particle process for which we
could prove convergence to a hydrodynamic equation and convergence of
fluctuations to a system of linear SDEs with a matrix $A$ satisfying
the conditions in Proposition \ref{prop:prop}, we would have
automatically that the direction of slow decay of correlations would
coincide with the characteristic lines of the PDE.

\section{Correlations of the Gaussian system}

In this section we study the space-time correlations of the SDE system
\eqref{eq:2}. Since the equations are linear, they
can be solved explicitly.

We formulate the results in wider generality. Again, the solution
to \eqref{eq:2}  with initial condition $\bar \xi$ is denoted $\{\xi_{p,t}\}_{t\ge0, p\in \cR_m}$  and we let $\mathbb P_{\bar \xi}$ be its law. The matrix $A=\{A_{p,p'}\}_{p,p'\in \cR_m}$  is not necessarily given by \eqref{eq:A}
but is required to satisfy
 properties (1)--(5) of Proposition \ref{prop:prop}.
Theorems \ref{th:2} and \ref{th:muinv} hold in this generality.

Let $\mathcal C^{\bar \xi}(t,s)$ be the $m^2\times m^2$ symmetric matrix
\begin{gather}
  \label{eq:3}
  \mathcal C_{p,p'}^{\bar \xi}(t,s)=\mathbb E_{\bar \xi}
  [\xi_{p,t}\xi_{p',s}], \quad p,p'\in \cR_m
\end{gather}
and $\mathcal M^{\bar\xi}(t)\in \mathbb R^m$ be the vector
\begin{gather}
  \mathcal M_p^{\bar \xi}(t)=\mathbb E_{\bar \xi} [\xi_{p,t}], \quad p\in \cR_m
\end{gather}
with of course $\mathcal C^{\bar \xi}_{p,p'}(0,0)=\bar\xi_p\bar\xi_{p'}$,
$\mathcal M^{\bar \xi}_{p}(0)=\bar\xi_p$.
By It${\rm \bar o}$'s lemma,
\begin{gather}
  \label{eq:4}
\frac d{dt}\mathcal M_p^{\bar \xi}=\sum_{p'} A_{p,p'}\cM^{\bar \xi}_{p'}(t),\\
\nonumber
  \frac d{dt}\mathcal C^{\bar \xi}(t,t)=v\mathbb I +\mathcal C^{\bar
    \xi}(t,t)A^T+A \mathcal C^{\bar \xi}(t,t),\\
\nonumber
  \frac d{dt}\mathcal C(t,s)=A \mathcal C(t,s),\quad t>s,
\end{gather}
with $A^T$ the transpose of $A$ and $\mathbb I$ the identity matrix.

Let also
\begin{gather}
  \label{eq:21}
  \mathcal W^{\bar \xi}_{p,p'}(t,s):=\mathcal C_{p,p'}^{\bar
  \xi}(t,s)-\cM^{\bar \xi}_p(t) \cM^{\bar \xi}_{p'}(s). 
\end{gather}

\begin{Theorem}
\label{th:2}  Let $A$ satisfy the properties of Proposition \ref{prop:prop}. Then,  $\mathcal W^{\bar \xi}_{p,p+y}$ does not depend on  $\bar \xi$ or $p$  and the  limit
\begin{eqnarray}
  \label{eq:WW}
  \mathbb W_{y}(t,s)= \lim_{m\to\infty}\cW^{\bar\xi}_{p,p+y}(t,s)
\end{eqnarray}
exists for any given $y\in\mathbb Z^2$. Moreover,
\begin{eqnarray}
  \label{eq:44}
  \mathbb W_{y}(t,s)=
 \frac v{4\pi w}\int_{1+(t-s)/2}^{1+(t+s)/2}\frac{e^{-|H|^2/(4a)}}a da
+ j(t,s,y),
\end{eqnarray} where  $w>0$ is defined in \eqref{eq:38}, $H=V y+(t-s)V U$, $V$ is defined in \eqref{eq:Vv}, $U=i\nabla \hat A(0)$, and
$j$ satisfies
\begin{eqnarray}
  \label{eq:condizionij}
  \sup_{t,s,y}|j(t,s,y)|<\infty, \quad \lim_{\max(t-s,|y|)\to\infty}j(t,s,y)=0.
\end{eqnarray}

\end{Theorem}
If the matrix $A$ is the one given in \eqref{eq:A}, using \eqref{eq:38} and \eqref{eq:1}, one sees that
\begin{eqnarray}
  \label{eq:sees}
 \frac v{ w}
= {\sqrt{e^{D}-1}}.
\end{eqnarray}

From \eqref{eq:44} one can obtain all desired asymptotics. For instance,
one can obtain equal-time correlations.
\begin{Corollary}[Equal-time correlations]
\label{cor1}
For $y=0$ one has
\begin{eqnarray}
  \label{eq:44bis}
  \limsup_{t\to\infty}\left|\mathbb W_{0}(t,t)- \frac v{4\pi w}\log t\right|<\infty.
\end{eqnarray}
For $y\ne0$ one finds, with $Y=Vy$ and after the change of variables $a=|Y|^2/(4x)$,
\begin{gather}
 \mathbb W_{y}(t,t)= \frac v{4\pi w}\int_{|Y|^2/4(1+t)}^{|Y|^2/4}\frac{e^{-x}}x\;dx
+j(t,t,y)  .
\end{gather}
This implies that
\begin{eqnarray}
  \label{eq:44ter}
\lim_{\substack {|y|\to\infty, t\to\infty\\ |y|=O(\sqrt t)}}
\left(\bW_{y}(t,t)-\frac v{4\pi w}\log(4(t+1)/|Y|^2)\right)=0
\end{eqnarray}
where we used the fact that
\begin{eqnarray}
  \label{eq:c-c+}
c_-\le \frac{|Y|}{|y|}\le c_+
\end{eqnarray}
 for some non-zero constants $c_\pm$ (the lower bound holds because the determinant of $V$ is not zero).

If instead  both $|y|,t$ diverge and $|y|\gg \sqrt t$, then $ \mathbb W_{y}(t,t)=o(1)$.

The same results hold if $ t-s=O(1)$.
\end{Corollary}

As for correlations at different times, the following result shows
that the behavior is special  along the space-time lines $y=t U$ (that
will be called ``characteristics'', in view of the discussion in
Section \ref{sec:pde}), with $U=i\nabla \hat A(0)=\sum_p p A_{p,0} \in \mathbb R^2$. We will assume that $t-s\gg1$, since  the case $t-s=O(1)$ is effectively covered by the previous Corollary.
\begin{Corollary}[Correlations along the characteristics]
\label{cor2}
We deduce from \eqref{eq:44}
\begin{gather}
\label{eq:along}
\limsup_{t-s\to\infty}\left(\bW_{\lfloor U(t-s)\rfloor}(t,s)
  -\frac v{4\pi w}\log\frac{t+s}{t-s}\right)=0,
\end{gather}
in particular $\bW_{\lfloor U(t-s)\rfloor}(t,s)$
 is large if $t-s\ll t$.

If instead  $u\ne U$, we have
\begin{gather}
\limsup_{t-s\to\infty}\left(\bW_{\lfloor u(t-s)\rfloor}(t,s)
  -\frac v{4\pi w}\int_{\frac{(t-s)^2|V(U-u)|^2}{2(t+s)}}^\infty \frac{e^{-x}}xdx\right)=0;
\end{gather}
from this we deduce that
\begin{eqnarray}
  \label{eq:nantra}
  \limsup_{t\to\infty,(t-s)/\sqrt t\to\infty}\bW_{\lfloor u(t-s)\rfloor}(t,s)=0,
\end{eqnarray}
 while
\begin{gather}
\limsup_{t-s\to\infty, t-s=O(\sqrt t)}\left|\bW_{\lfloor u(t-s)\rfloor}(t,s)- \frac v{4\pi w} [\log t-2\log(t-s)]\right|
<\infty.
\end{gather}

\end{Corollary}

Summarizing: along the characteristic correlations are large as soon
as $(t-s)\ll t$ and grow proportionally to $\log t$ if $t-s\le
t^a,a<1$. For all other space-time directions, correlations are large only if $t-s\ll \sqrt t$.

Equation \eqref{eq:along} shows that, along the characteristic, there is a limit for the
correlation as $t,s\to\infty$  with $t-s$ and $t+s$ of comparable size. This
suggests that in this scaling  the whole fluctuation field near the characteristic  has a non-trivial limit process which, in the following statement,
we identify as being related to the Stochastic Heat Equation.
\begin{Corollary}
\label{cor3}
Define
\begin{eqnarray}
  \label{eq:43}
  h_{x,t}=a\xi_{\lfloor t U+V^{-1}x\rfloor ,t},\quad
  a=\sqrt{\frac{4\pi w}{8v}}.
\end{eqnarray}
Then,  for any given $x,y,$ $0<s<t$, from \eqref{eq:44} we get
\begin{gather}
  \label{eq:45}
\lim_{\delta\to0}\left[ \mathbb E_{\bar\xi}(h_{\delta^{-1/2}x,\delta^{-1}t}\;h_{\delta^{-1/2}y,\delta^{-1}s})-  \mathbb
  E_{\bar\xi}(h_{\delta^{-1/2}x,\delta^{-1}t})\mathbb E_{\bar \xi}(h_{\delta^{-1/2}y,\delta^{-1}s})\right]\\=
\frac 18\int_{(t-s)/2}^{(t+s)/2}\frac{e^{-|x-y|^2/(4a)}}a da
.
\end{gather}
In other words, in this limit the randomly evolving height field $h_{\cdot,\cdot}$ has
the same space-time correlations as the
additive stochastic heat equation in $2$ spatial dimensions,
\begin{eqnarray}
  \label{eq:53}
  \partial_t h=\Delta h+ \dot W
\end{eqnarray}
with $\dot W$ the $(2+1)$-dimensional space-time white noise
 (compare \eqref{eq:45} with the formal expression derived in \cite[Formula
(2.8)]{Hairer} for the covariance of the stochastic heat equation).
\end{Corollary}

\begin{Remark}
  If the particle label $p$ is $d$-dimensional, $ d\ne 2$, and the
  matrix $A$ satisfies the analog the properties (1)--(5) of
  Proposition \ref{prop:prop} (with $[-\pi,\pi]^2$ replaced by
  $[-\pi,\pi]^d$), then one can check that Theorem \ref{th:2} still
  holds, with \eqref{eq:44} replaced by
\begin{eqnarray}
  \label{eq:dnon2}
   \mathbb W_{y}(t,s)=
 \frac {\pi v}{(2\pi)^d w}\int_{1+(t-s)/2}^{1+(t+s)/2}\frac{e^{-|H|^2/(4a)}}{a^{d/2}} da
+ j(t,s,y), \quad y\in \mathbb Z^d.
\end{eqnarray}
We leave it to an interested reader 
to derive the analogs of Corollaries \ref{cor1}--\ref{cor3} for $d\ne 2$.
\end{Remark}

\subsection{Stationary measure}

 We cannot expect that there is a stationary measure for $\{\xi_{p}\}_p$: since there is invariance by global shifts on $\mathbb R$ of all positions $\xi_p$, the inverse
covariance matrix will have a zero mode (like a Gaussian Free Field
not pinned to zero anywhere). However, the stationary measure on
gradients of $\xi$ will be well-defined. Not surprisingly, its scaling
limit (which requires letting $m\to\infty$) is the standard massless Gaussian Free Field
(see e.g. \cite{GFF} for definitions), up to an  affine
transformation of coordinates by the matrix $V$ (cf. \eqref{eq:varianza}).

\begin{Theorem}
\label{th:muinv}
Let $A$ satisfy the properties of Proposition \ref{prop:prop}.
There exists a unique  stationary measure $\mu$ for the gradients $(\xi_p-\xi_{p'})_{p,p'\in\cR_m}$. This measure  is Gaussian and
  its mean and covariances are given by 
  \begin{eqnarray}
\label{eq:nomedia}
    \mathbb E_\mu(\xi_p-\xi_{p'})=0 \quad \forall p,p'
  \end{eqnarray}
and
\begin{eqnarray}
  \label{eq:covar}
      {\rm Cov}_\mu\left[
(\xi_{y_1}-\xi_{y_2});(\xi_{y_3}-\xi_{y_4})\right]&=&-\frac
  v{m^2}\sum_{k\in \mathcal K_m}\frac{(e^{i
  k y_1}-e^{i k y_2})(e^{-i k y_3}-e^{-i k y_4})}{\hat R(k)}
\\&
\stackrel{m\to\infty}\to&
\label{eq:moreg}
-\frac{v}{(2\pi)^2}\int_{[-\pi,\pi]^2}d k
\frac{(e^{i
  k y_1}-e^{i k y_2})(e^{-i k y_3}-e^{-i k y_4})}{\hat R(k)}\\&\equiv&{\rm Cov}_{\mu_\infty}\left[
(\xi_{y_1}-\xi_{y_2});(\xi_{y_3}-\xi_{y_4})\right].
\end{eqnarray}
Moreover, one has
\begin{gather}
  \label{4punti}
{\rm Cov}_{\mu_\infty}[(\xi_{y_1}-\xi_{y_2});(\xi_{y_3}-\xi_{y_4})]=
\frac v{2\pi w}\log\frac{1+|V(y_1-y_4)||V(y_3-y_2)|}{1+|V(y_1-y_3)||V(y_2-y_4)|}+R_{y_1,y_2,y_3,y_4}
\end{gather}
where $V$ is the matrix in \eqref{eq:Vv} 
and\footnote{here, $|y_i-y_j|$ denotes the Euclidean distance between
  $y_i$ and $y_j$ on $\mathbb Z^2$ and not on the ``torus'' $\cR_m$ (recall that the limit $m\to\infty$ has already been taken).}
\begin{eqnarray}
  \label{eq:R4}
R_{y_1,y_2,y_3,y_4}=O(1/(1+
\min(|y_1-y_3|,|y_2-y_4|,|y_1-y_4|,|y_2-y_3|))).
\end{eqnarray}

From this we deduce a convergence to a massless GFF on $\mathbb R^2$, of covariance proportional to
\[
-\log |V(x-y)|,
\]
in the following sense: Let $\phi:\mathbb R^2\mapsto \mathbb R$ be a $C^\infty$, compactly supported function such that $\int_{\mathbb R^2}\phi(x)dx=0$. Then, the zero-average random function
  \begin{eqnarray}
    \label{eq:smooth}
 \xi_\phi:=   \delta^2\sum_{p}\phi(\delta p)(\xi_p-\xi_0)
  \end{eqnarray}
converges in law, in the limit $\lim_{\delta\to0}\lim_{m\to\infty}$, to
a centered Gaussian random variable of variance
\begin{eqnarray}
  \label{eq:varianza}
- \frac v{2\pi w} \int \phi(x)\phi(y)\log |V(x-y)|dx dy.
\end{eqnarray}
\end{Theorem}
Remark that, given $\phi^{(i)},i=1,2$ satisfying the same assumptions as $\phi$ above, the limit covariance $\lim_{\delta\to0}\lim_{m\to\infty}\mathbb E_\mu(\xi_{\phi^{(1)}}\xi_{\phi^{(2)}})$ can be simply deduced via
\begin{eqnarray}
2  \mathbb E_\mu(\xi_{\phi^{(1)}}\xi_{\phi^{(2)}})=\mathbb E_\mu(\xi^2_{\phi^{(1)}})+\mathbb E_\mu(\xi_{\phi^{(2)}}^2)
-\mathbb E_\mu((\xi_{\phi^{(1)}-\phi^{(2)}})^2).
\end{eqnarray}

In particular, \eqref{4punti} gives
\begin{gather}
  \label{eq:39}
  \limsup_{|y|\to\infty}   \left|{\rm Var}_{\mu_\infty}[\xi_0-\xi_y]-
\frac v{\pi w}\log |V y|\right|<\infty
\end{gather}
(using  \eqref{eq:c-c+}, one can replace $Vy$ with $y$).

\section{Convergence to the SDEs: Proof of Theorem \ref{th:SDE}}
\label{sec:t1}

\subsection{A tightness estimate}

Let
\begin{eqnarray}
  \label{eq:G}
  \mathcal G=\left\{\sigma\in\Omega_{L,N;m_1,m_2}:\forall p, \left|B_p-\frac B\epsilon\right|\le \frac{\log (1/\epsilon)}{\sqrt\epsilon}, \left|D_p-\frac D\epsilon\right|\le \frac{\log (1/\epsilon)}{\sqrt\epsilon}
\right\}.
\end{eqnarray}
Note that the initial condition $\sigma_0$ is well inside $\mathcal G$, cf. \eqref{eq:xi}.
The crucial ingredient in the proof of Theorem \ref{th:SDE} is the following a-priori tightness estimate:
\begin{Lemma}
\label{lemma:apriori}
  Let $\partial \mathcal G\subset \mathcal G$ denote the inner boundary of $\mathcal G$ (the set
  of configurations from which the dynamics can exit $\mathcal G$ with a single update) and define the stopping time
  \begin{gather}
    \tau_{\mathcal G}=\inf\{t>0: \sigma(t)\in\partial \mathcal G\}.
  \end{gather}
Then,
\begin{gather}
\label{eq:taulargo}
\lim_{\epsilon\to0} \mathbb P_{\sigma_0}(\tau_{\mathcal G}\le \epsilon^{-2})=0.
\end{gather}
\end{Lemma}
The time $\epsilon^{-2}$ could be replaced by any
$\epsilon^{-a}$. What we need is that $a>1$, so that this time is much
larger than $1/\epsilon$, the time-scale of convergence to the SDEs.

\begin{proof}[Proof of Lemma \ref{lemma:apriori}]
  Let us start with the following:
  \begin{Lemma}
\label{lemma:qq}
Let $a=\epsilon^{-1}b+X$ with  $b>0$ and  $\sqrt \epsilon |X|\ll \epsilon^{-1/10}$.
The following asymptotic expansion holds as $\epsilon\to0$:
\begin{gather}
  (q;q)_a=\exp\left[
\frac1\epsilon\sum_{n\ge1}\frac1{n^2}e^{-bn}-\frac12\sum_{n\ge1}\frac1ne^{-bn}\right.
\\
\left.- X\sum_{n\ge1}\frac1n e^{-bn}+\epsilon\frac{X^2}2\frac{e^{-b}}{1-e^{-b}}+C(\epsilon)+o(1)
\right]
\end{gather}
with $C(\epsilon)$ independent of $b $ and $X$.
  \end{Lemma}
See Appendix \ref{app:qq} for the proof.

\begin{Lemma}
\label{th:locmax}
If $\sigma\in{\mathcal
  G}$ and if as in \eqref{eq:xi} we let
$\eta_p:=\sqrt\epsilon(x_p-\epsilon^{-1}P_p)$, then
\begin{gather}
  \label{eq:11}
  \pi(\sigma)=\frac1Z \exp\left(v\sum_{k\in\mathcal K_m} \hat R(k)|\hat {\eta_k}|^2+o(1)\right)
\end{gather}
where $\pi(\cdot)$ is defined in \eqref{eq:3prime} and, as in \eqref{eq:hatxi}, we set
\begin{eqnarray}
  \label{eq:hateta}
  \hat {\eta_k}=\sum_{p\in\cR_m}\eta_p f_k(p),\quad \eta_p=\sum_{k\in\mathcal K_m}\hat \eta_k\overline{f_k(p)}.
\end{eqnarray}
We recall from Proposition \ref{prop:prop} that $\hat R(k)\le 0 $ and vanishes only for $k=0$.
\end{Lemma}
\begin{proof}[Proof of Lemma \ref{th:locmax}]
  From Lemma \ref{lemma:qq} we see that
  \begin{eqnarray}
    \label{eq:46}
     \pi(\sigma)=\frac1Z \exp\left(-\frac12 (\eta,Q \eta)+o(1)\right)
  \end{eqnarray}
with $Z$ a normalization constant independent of $\sigma$ and
\begin{gather}
  - (\eta, Q \eta)=\sum_p
  (\eta_p-\eta_{p+(1,0)})^2\frac{e^{-D}}{1-e^{-D}}\\-
\sum_p (\eta_p-\eta_{p+(1,-1)})^2\frac{e^{-B}}{1-e^{-B}}-
\sum_p (\eta_p-\eta_{p+(0,-1)})^2\frac{e^{-C}}{1-e^{-C}}.
\end{gather}
This can be rewritten in Fourier space as
\begin{gather}
  \label{eq:Fourier}
  -\frac12(\eta, Q \eta)=\sum_k |\hat \eta_k|^2\hat Q(k)\\
\hat Q(k)=
\frac{e^{-D}}{1-e^{-D}}(1-\cos(k_1))-\frac{e^{-B}}{1-e^{-B}}(1-\cos(k_1-k_2))
\\-\frac{e^{-C}}{1-e^{-C}}(1-\cos(k_2)).
\end{gather}
Here, we use for instance that
\begin{gather}
  \sum_p(\eta_p-\eta_{p+(1,0)})^2=\sum_p\sum_{k,k'\in\mathcal K_m}\hat \eta_k\hat \eta_{-k'}\overline{f_k(p)}f_{k'}(p)(1-e^{i k_1})(1-e^{-i k'_1})\\
=\sum_{k\in\mathcal K_m}|\hat \eta_k|^2(2-2\cos(k_1))
\end{gather}
where we used orthonormality of $\{f_k(\cdot)\}_{k\in\mathcal K_m}$ and the second of \eqref{eq:eta-eta}.
One checks that
\begin{eqnarray}
\label{eq:QAR}
  \hat Q(k)=v (\hat A(k)+\hat A(-k))=v\hat R(k)
\end{eqnarray}
and the proof is concluded.
\end{proof}

Given any Markov chain with stationary measure $\pi$, generator $\mathcal L$  and two states
$x,y$, we have for any $t>0$
\begin{eqnarray}
  \label{eq:5}
P_t(x,y)\le \frac{\pi(y)}{\pi(x)}.
\end{eqnarray}
To see this, write
\begin{eqnarray}
  \label{eq:7}
  P_t(x,y)=[\delta_x e^{t\mathcal L}](y)\le \left[\frac{\pi}{\pi(x)}e^{t\mathcal L}\right](y)=\frac{\pi(y)}{\pi(x)}
\end{eqnarray}
where we used $\pi(\cdot)/\pi(x)\ge \delta_x(\cdot)$.
For the initial configuration, one has from \eqref{eq:46}
\begin{eqnarray}
  \label{eq:8}
  \pi(\sigma_0)\ge \frac1Z e^{-K_1}
\end{eqnarray}
for some finite constant $K_1$ depending on $\bar\xi$ (and, typically, of order $m^2$).
If instead $\sigma\in\partial \mathcal G$, then one has
\begin{gather}
\label{eq:WR}
  \pi(\sigma)\le \frac1Z e^{-K_2(\log \epsilon)^2}
\end{gather}
for some strictly positive $K_2$ (this is proven below).
Finally, observe that
\begin{eqnarray}
  \label{eq:ng}
  |\partial \mathcal G| = O(\epsilon^{-K_3})
\end{eqnarray}
 for some constant
$K_3$ depending on $m$. This is trivial: indeed, the total number of configurations is bounded by $(\ell/\epsilon)^{m^2}$ (recall the definition of the model,  the fact that $\ell$
is independent of $\epsilon$ and that there are $m^2$ particles).
Using \eqref{eq:5}, \eqref{eq:8}, \eqref{eq:WR} and \eqref{eq:ng}, we deduce
\begin{eqnarray}
  \label{eq:PGc}
  \mathbb P(\sigma(t)\in\partial\mathcal G)\le \exp(-K_4(\log \epsilon)^2).
\end{eqnarray}
As a consequence,
\begin{eqnarray}
  \label{eq:34}
  \mathbb E\left[\int_0^{\epsilon^{-2}+1}1_{\{\sigma(t)\in\partial\mathcal
  G\}}\,dt\right]\le (\epsilon^{-2}+1) \exp(-K_4(\log \epsilon)^2).
\end{eqnarray}
Next observe that, if $\tau_{\mathcal G}<\epsilon^{-2}$, then there exists a
probability at least $\delta>0$ independent of $\epsilon$ that the total time spent in  $\partial\mathcal
  G$ up to time $\epsilon^{-2}+1$ is at least $\delta\times\epsilon$. This is simply
  because the maximal transition rate of the Markov chain is of order
  $\epsilon^{-1}$
(this is the case when $C_p$ in \eqref{eq:6prime} is of order $1$).
In conclusion,
\begin{eqnarray}
  \label{eq:35}
   \mathbb E\left[\left.\int_0^{\epsilon^{-2}+1}1_{\sigma(t)\in\partial\mathcal
  G}dt\right|\tau_{\mathcal G}\le \epsilon^{-2}\right]\ge \delta^2\epsilon.
\end{eqnarray}
Together with \eqref{eq:34} we deduce
\begin{eqnarray}
  \label{eq:36}
  \mathbb P(\tau_{\mathcal G}\le \epsilon^{-2})\le \frac1{\delta^2\epsilon}\mathbb E\left[\int_0^{\epsilon^{-2}+1}1_{\{\sigma(t)\in\partial\mathcal
  G\}}\,dt\right]\le \exp(-K_5(\log \epsilon)^2)
\end{eqnarray}
and
\eqref{eq:taulargo} is proven.

It remains to prove \eqref{eq:WR}.
Recall  that  $(\hat A(k)+\hat A(-k))$ is negative for
every $k\in\mathcal K_m$ and vanishes only for $k=0$, so that
\begin{eqnarray}
  \label{eq:29}
  -\frac12(\eta, Q \eta)\le -\delta\sum_{k\in\mathcal K_m,k\ne0} |\hat \eta_k|^2
\end{eqnarray}
for some positive $\delta$ that depends only on the number of
particles $m^2$. 

If $\sigma\in\partial\mathcal G$
 we must have either
$|\eta_p-\eta_{p+(1,0)}|\ge (1/2)\log(1/\epsilon)$ or
$|\eta_p-\eta_{p+(1,-1)}|\ge (1/2)\log(1/\epsilon)$
for some $p$. Assume to fix ideas that the former is the case. Then, from \eqref{eq:hateta} (and writing $\sum_k$ instead of $\sum_{k\in\mathcal K_m}$)
\begin{gather}
  \label{eq:33}
  \frac12 \log(1/\epsilon)\le |\eta_p-\eta_{p+(1,0)}|=
|\sum_{k}\hat \eta_k \overline{ f_k(p)}(e^{i k_1}-1)|\\=
|\sum_{k\ne0}\hat \eta_k \overline{f_k(p)}(e^{i k_1}-1)|\le2
\sqrt{\sum_{k\ne0}|\hat \eta_k |^2}
\sqrt{\sum_k|f_k(p)|^2}= 2 \sqrt{\sum_{k\ne0}|\hat \eta_k |^2}.
\end{gather}
Then, \eqref{eq:WR} immediately follows.
\end{proof}

\subsection{Proof of Theorem \ref{th:SDE}}

Let $U(t)=\{U_p(t)\}_{p\in\cR_m}$ be defined as
\begin{eqnarray}
  \label{eq:Bp}
  U_p(t)=\frac1{\sqrt\epsilon}\int_0^t\left[-v+ r_p(\sigma(s/\epsilon)\right]ds
\end{eqnarray}
where $r_p(\sigma)$ is the rate at which particle $p$ jumps $+1$ to the right in the configuration $\sigma$
(this includes the event that it jumps because it is pushed by another
particle, i.e. because clock of particle $p'$ rings and $p\in V^+_{p'}$).
Note that $\{M_p(t)\}_t\equiv \{\eta_{p,t}-U_p(t)\}_t$ is a martingale,
since we have
\begin{eqnarray}
  \label{eq:29bis}
  \frac d{dt}U_p(t)=\left.\partial_s\mathbb E[\eta_{p,t+s}|\{\eta_{p',t}\}_{p'}]\right|_{s=0^+}.
\end{eqnarray}
Next, define $V(t)=\{V_{p,p'}(t)\}_{p,p'}$ as solution to
\begin{eqnarray}
  \label{eq:Vpq}
\frac d{dt}  V_{p,p'}(t)=\left.\partial_s\mathbb E[M_{p,t+s}M_{p',t+s}|\{\eta_{r,t}\}_r]\right|_{s=0^+}\quad V_{p,p'}(0)=0,
\end{eqnarray}
so that $\{M_p(t)M_{p'}(t)-V_{p,p'}(t)\}_t$ is again a martingale for every $(p,p')$.

We will apply  \cite[Theorem 4.1, Chapter 7]{cf:EK}, that gives a set of sufficient conditions on the processes $\eta, U$ and $V$ that imply
that $\{\eta_{p,t}\}_{t\ge0,p\in\cR_m}$ converges weakly to the solution of \eqref{eq:2}\footnote{A warning on notations: in \cite{cf:EK}, $\eta$ is called $X$, $U$ is called $B$ and $V$ is called $A$}.
In particular, conditions (4.3)--(4.5) in \cite{cf:EK} are trivial.
For (4.6), we have to check that, for any $T>0$,
\begin{eqnarray}
  \label{eq:4.6}
  \sup_{t\le T}\left|
U_p(t)-\int_0^t\sum_{p'} A_{p,p'}\eta_{p',s}ds
\right|\stackrel {\mathbb P} \to0
\end{eqnarray}
as $\epsilon\to0$.
On the event $\{\tau_{\mathcal G}>\epsilon^{-2}\}$, whose probability tends to $1$ as $\epsilon\to0$ thanks to Lemma \ref{lemma:apriori},
  no particle can push any other in configuration $\sigma(s/\epsilon)$
  for any $s/\epsilon\le \epsilon^{-2}$ (particles are far away from each
  other and all the families $V^+_p$ defined just before
  \eqref{eq:6prime}) include only the particle
  $p$ itself).
Then, $r_p(\sigma(s/\epsilon))$ equals the r.h.s. of \eqref{eq:6prime}
for any $s\le T$.
We have
\begin{gather}
\label{eq:circav}
  r_p(\sigma(s/\epsilon))=\frac{(1-e^{-B-\sqrt\epsilon(\eta_{p+(1,-1),s}-\eta_{p,s})})
(1-e^{-D-\sqrt\epsilon(\eta_{p,s}-\eta_{p-(1,0),s})-\epsilon})}
{1-e^{-C-\sqrt\epsilon(\eta_{p,s}-\eta_{p-(0,1),s})-\epsilon}}\\=
v+\sqrt\epsilon\left[e^{-B}\frac{1-e^{-D}}{1-e^{-C}}(\eta_{p+(1,-1),s}-\eta_{p,s})
+e^{-D}\frac{1-e^{-B}}{1-e^{-C}}(\eta_{p,s/\epsilon}-\eta_{p-(1,0),s})\right.\\
\left.-e^{-C}\frac{(1-e^{-B})(1-e^{-D})}{(1-e^{-C})^2}(\eta_{p,s}-\eta_{p-(0,1),s})
\right]+O(\epsilon\log(1/\epsilon))\\=
v+\sqrt\epsilon\sum_{p'} A_{p,p'}\eta_{p',s}+O(\epsilon\log(1/\epsilon))
\end{gather}
where $O(\epsilon\log(1/\epsilon))$ is uniform on $s\le T$.
We used that, on $\mathcal G$, $|\eta_{p-b}-\eta_p|\le \log(1/\epsilon)$, $b\in\{(-1,1),(1,0),(0,1)\}$.
Plugging this into \eqref{eq:Bp}  we get \eqref{eq:4.6}.

Finally, condition (4.7) in \cite[Theorem 4.1, Chapter 7]{cf:EK} amounts in our case to requiring that
\begin{eqnarray}
  \label{eq:4.7}
   \sup_{t\le T}\left|V_{p,p'}(t)-vt \delta_{p,p'}\right|\stackrel {\mathbb P} \to0.
\end{eqnarray}
Indeed,
one has (using that $\{M_{p,t}\}_{t\ge0}$ is a martingale)
\begin{gather}
  \frac d{dt} V_{p,p'}(t)= \partial_s\mathbb E\left[
N_{p,t+s}N_{p',t+s}
|\{\eta_{r,t}\}_r
\right]|_{s=0^+}
\end{gather}
where
\[
N_{p,t+s}=\sqrt \epsilon x_p(\epsilon^{-1}(t+s))-\frac1{\sqrt \epsilon}\int_0^{t+s}r_p(\sigma(\epsilon^{-1} u))du.
\]
Note that, by definition of the jump rate $r_p(\sigma)$,
\begin{gather}
  \partial_s\mathbb E\left[\left. x_p(\epsilon^{-1}(t+s))\int_0^{t+s}r_{p'}(\sigma(\epsilon^{-1} u))du\right|\{\eta_{r,t}\}_r\right]_{s=0^+}\\=\epsilon^{-1}r_p(\epsilon^{-1}t)\int_0^tr_{p'}(\sigma(\epsilon^{-1}u)du+x_p(\sigma(\epsilon^{-1}t)r_{p'}(\sigma(\epsilon^{-1}t))
.
\end{gather}
Therefore, one sees that
\begin{gather}
  \label{eq:30}
 \frac d{dt} V_{p,p'}(t)=-x_p( t/\epsilon)r_{p'}(\sigma( t/\epsilon))
-x_{p'}(t/\epsilon)r_p(\sigma(t/\epsilon))\\
+\epsilon\left.\partial_s\mathbb
  E[x_p((t+s)/\epsilon)x_{p'}((t+s)/\epsilon)|\{\eta_{r,t}\}_{r}]\right|_{s=0^+}.
\end{gather}
If $p=p'$, this gives
\begin{gather}
  \label{eq:31}
  \frac d{dt} V_{p,p'}(t)=-2x_p( t/\epsilon)r_p(\sigma( t/\epsilon))
+r_p(\sigma( t/\epsilon))(1+2x_p(t/\epsilon))\\=r_p(\sigma(t/\epsilon))
\end{gather}
where in the first step we used the fact that when $x_p$ jumps $+1$ (which happens with rate $r_p$), $x_p^2$ increases by $(x_p+1)^2-x_p^2=1+2x_p$.
Recall that, on the event $\{\tau_{\mathcal G}\ge \epsilon^{-2}\}$, we have
$r_p(\sigma(t/\epsilon))=v+o(1)$, cf. \eqref{eq:circav}. Since the probability
of $\{\tau_{\mathcal G}\ge \epsilon^{-2}\}$ tends to $1$, this implies
\eqref{eq:4.7} for $p=p'$.
As for $p\ne p'$,
on the event $\{\tau_{\mathcal G}\ge \epsilon^{-2}\}$ the particles $p$
and $p'$ cannot jump simultaneously since all particles are well
spaced all the time so that no particle can push any other. Then, on this event,
\begin{gather}
  \label{eq:32}
  \epsilon\left.\partial_s\mathbb
  E[x_p((t+s)/\epsilon)x_{p'}((t+s)/\epsilon)|\{\eta_{r,t}\}_{r}]\right|_{s=0^+}\\=r_{p'}(\sigma(t/\epsilon))x_p(t/\epsilon)+
r_p(\sigma(t/\epsilon))x_{p'}(t/\epsilon)
\end{gather}
so that
$\frac d{dt}V_{p,{p'}}(t)=0$ for every $t\le T$ and
\eqref{eq:4.7} follows.

\section{Space-time correlations of the SDEs}

\subsection{Proof of Theorems \ref{th:2}}

In Fourier space, Eqs.  \eqref{eq:4} give
\begin{gather}
\label{eq:dxik}
\frac d{dt}\mathbb E_{\bar \xi}(\hat \xi_{k,t})=\hat A(k) \mathbb E_{\bar \xi}(\hat \xi_{k,t}),\\
  \label{eq:xixi}
  \frac d{dt} \mathbb  E_{\bar \xi}(\hat \xi_{k,t} \hat \xi_{k',t})=
v\delta_{k+k'=0}+\mathbb  E_{\bar \xi}(\hat \xi_{k,t} \hat \xi_{k',t})
(\hat A(k')+\hat A(k))\\
\label{eq:xixits}
\frac d{dt} \mathbb  E_{\bar \xi}(\hat \xi_{k,t} \hat \xi_{k',s})=\hat
A(k) \mathbb E_{\bar \xi}(\hat \xi_{k,t} \hat \xi_{k',s}), \quad t>s
\end{gather}
whose solution is (just differentiate w.r.t. $t$ to check)
\begin{gather}
  \label{eq:47}
  \mathbb E_{\bar \xi}(\hat \xi_{k,t})=e^{\hat A(k)t}\hat\xi_{k,0}\\
\mathbb E_{\bar\xi}(\hat\xi_{k,t}\hat \xi_{k',t})=v
\delta_{k=-k'}\frac{e^{\hat R(k)t}-1}{\hat
  R(k)}+\mathbb E_{\bar \xi}(\hat \xi_{k,t}) \mathbb E_{\bar \xi}(\hat
\xi_{k',t}),
\\
\label{aterza}
\mathbb E_{\bar\xi}(\hat\xi_{k,t}\hat \xi_{k',s})=v
\delta_{k=-k'}e^{\hat A(k)(t-s)}\frac{e^{\hat R(k)s}-1}{\hat
  R(k)}+\mathbb E_{\bar \xi}(\hat \xi_{k,t}) \mathbb E_{\bar \xi}(\hat
\xi_{k',s}), \quad t\ge s.
\end{gather}
Note that it immediately follows that $\mathbb  E_{\bar \xi}(\hat \xi_{k,t} \hat \xi_{k',s})-\mathbb E_{\bar \xi}(\hat \xi_{k,t}) \mathbb E_{\bar \xi}(\hat
\xi_{k',s})$ (and therefore $\mathcal W^{\bar \xi}_{p,p'}(t,s)$) is independent of the initial condition $\bar\xi$.

Using the first of \eqref{eq:eta-eta} together with \eqref{aterza} we find, for $t\ge s$,
\begin{gather}
  \label{eq:24}
   \mathcal W^{\bar \xi}_{p,p+y}(t,s)=\frac v{m^2}\sum_{k\in\mathcal K_m} e^{-i k
     y}e^{\hat A(k)(t-s)}\frac{e^{\hat R(k)s}-1}{\hat
  R(k)}\\
\label{aseconda}
\stackrel{m\to\infty}\to \frac v{(2\pi)^2}
\int_{[-\pi,\pi]^2}dk
\frac{e^{ \hat R(k)s}-1}{\hat R(k)}
e^{(t-s)\hat A(k)}e^{-i k y}.
\end{gather}
Observe that, as $m\to\infty$, the set $\mathcal K_m$ fills the
parallelogram
\[\mathcal K_\infty:=\{(k_1,k_2)\in \mathbb R^2:| k_1|\le\pi,|k_2-(C/D)k_1|\le \pi\}.\]
However, since the integrand of \eqref{aseconda} is $2\pi$-periodic,
integrating over $\mathcal K_\infty$ or on $[-\pi,\pi]^2$ gives the
same result.

Recall property (3) in Proposition \ref{prop:prop}.
The singularity at $(0,0)$ is integrable and
the dominant contribution to the integral
comes from $k\sim 0$.
From \eqref{aseconda} one deduces \eqref{eq:44}, see
Appendix \ref{app:326} for details.

\subsection{Proof of Theorem \ref{th:muinv}}

The gradients $\xi_{p,t}-\xi_{p',t}$ can be written via \eqref{eq:eta-eta}
as linear combinations of the Fourier components $\{\hat \xi_{k,t}\}_{ k\in \mathcal K_m, k\ne 0}$. Since the random variables $\hat \xi_{k,t}$ solve a set of linear SDEs, their invariant measure $\mu$ (if it exists) is Gaussian. Stationarity of $\mu$ implies, via \eqref{eq:dxik} and \eqref{eq:xixi}, that for every $k,k' \in\mathcal K_m, k,k'\ne0$,
\begin{eqnarray}
  \label{eq:inv}
\mathbb E_\mu(\hat \xi_{k})&=&0\\
\mathbb E_\mu(  \hat \xi_{k} \hat \xi_{k'})&=&-\delta_{k=-k'}  \frac v{\hat A(k)+\hat A(-k)}.
\end{eqnarray}
where in the first equality we used that the fact that $\hat A(k)\ne0$ whenever $k\ne0$.
Since a Gaussian measure is uniquely characterized by mean and variance, uniqueness of the invariant measure also follows.
Using \eqref{eq:eta-eta} we see that \eqref{eq:nomedia} and \eqref{eq:moreg} hold.
Again, the dominant contribution comes from $k\sim0$ and one obtains
the asymptotics \eqref{4punti}, see Appendix
\ref{app:322}.

Next we prove the claim about the limit behavior of the random
variable $\xi_\phi$. Since $\xi_\phi$  is Gaussian and centered for any $\delta,m$, it suffices to prove that its variance converges to \eqref{eq:varianza}.
Note that the support $S_\phi$ of the function $\phi(\delta\cdot)$ has a diameter of order $\delta^{-1}\ll m$. Let $p_0$ be a point outside $S_\phi$, at distance of order $1/\delta$ from the origin. Write
\begin{eqnarray}
  \mathbb E_\mu(\xi_\phi^2)=\delta^4\sum_{p,p'}\phi(\delta p)\phi(\delta p')
\mathbb E_\mu((\xi_p-\xi_0)(\xi_{p'}-\xi_{p_0}))+o(1)
\end{eqnarray}
where we used
\begin{eqnarray}
  \label{eq:sumd}
\delta^2\sum_{p'}\phi(\delta {p'})=O(\delta)
\end{eqnarray}
(because the integral of $\phi$ is zero and the function is smooth) together with
${\rm Cov}_{\mu_\infty}((\xi_p-\xi_0);(\xi_0-\xi_{p_0}))\le  C\log (\delta^{-1})$
for $p$ in the support of  $\phi(\delta\cdot)$, which follows from
\eqref{eq:39} and Cauchy-Schwarz.
Using \eqref{4punti} we have
\begin{gather}
  \mathbb E_\mu(\xi_\phi^2)=\frac{v}{2\pi w}\delta^4\sum_{p,{p'}}\phi(\delta p)\phi(\delta
  {p'})\log\frac{|V(p-p_0)||V q|}{|V(p-{p'})||V p_0|}
+o(1)
\\
\label{acosa}=
\frac{v}{2\pi w}\delta^4\sum_{p,{p'}}\phi(\delta p)\phi(\delta
  {p'})\log\frac{|V\delta(p-p_0)||V\delta {p'}|}{|V\delta(p-{p'})||V \delta p_0|}
+o(1)
\end{gather}
where in the first step we used
\begin{gather}
  \|\phi\|_\infty^2\delta^4\sum_{p,{p'}\in S_\phi}\frac1{\min(1,|p-p_0|,|{p'}|,|p-{p'}|,|p_0|)}\stackrel{\delta\to0}=o(1)
\end{gather}
The sum in \eqref{acosa} can be written as
\begin{gather}
\label{asera}
- \frac{v}{2\pi w}\delta^4\sum_{p,{p'}}\phi(\delta p)\phi(\delta {p'})\log|V\delta(p-{p'})|
+o(1).
\end{gather}
Indeed, the terms proportional to $\log |V(p-p_0)|, \log|V {p'}| $ and $\log |V p_0|$ are independent of at least one of the two summation variables $p,{p'}$: then, using once more
 \eqref{eq:sumd} together with e.g. $|\log |V \delta {p'}||=O(\log(1/\delta))$, \eqref{asera} follows.

The sum in \eqref{asera} is the Riemann approximation of the convergent integral \eqref{eq:varianza}.

 \appendix

\section{Proof of Lemma \ref{lemma:qq}}
\label{app:qq}

Set $\delta=1/10$.
To get the asymptotics of $(q;q)_a$, write
\begin{gather}
\label{eq:a1}
  \log (q;q)_a=\sum_{i=1}^a\log (1-q^i)=-\sum_{n=1}^\infty\frac1n\sum_{i=1}^a q^{ni}=
-\sum_{n=1}^\infty\frac1n \frac{q^n}{1-q^n}(1-q^{na}).
\end{gather}
Then we recall that $a=\epsilon^{-1}b+X$ and that $q=e^{-\epsilon}$ and we write
\begin{gather}
  \log(q;q)_a=R_1+R_2:=-\sum_{n=1}^\infty\frac1n\frac{q^n}{1-q^n}(1-e^{-b n})+
\sum_{n=1}^\infty\frac1n\frac{q^n}{1-q^n}e^{-b n}(e^{-nX\epsilon}-1).
\end{gather}
We have, using $q^n/(1-q^n)\sim (1-\epsilon n/2)/(\epsilon n)$ for $\epsilon n\ll1$,
\begin{gather}
\label{eq:R1}
  R_1=C(\epsilon)+\sum_{n=1}^\infty\frac1n\frac{q^n}{1-q^n}e^{-b n}=
C(\epsilon)+\frac1\epsilon \sum_{n=1}^\infty\frac1{n^2}e^{-bn}-\frac12\sum_{n=1}^\infty
\frac{e^{-bn}}n+o(1)
\end{gather}
with $C(\epsilon)=-\sum_n (1/n)q^n/(1-q^n)$ independent of $b,X$.
As for $R_2$, we claim that
\begin{gather}
\label{eq:R2}
  R_2=-X\sum_{n=1}^\infty\frac{e^{-b n}}n+\epsilon \frac{X^2}2\sum_{n=1}^\infty e^{-bn}+o(1)
\end{gather}
which, together with \eqref{eq:R1}, concludes the proof of the Lemma.

To get \eqref{eq:R2}, remark first of all that
\begin{gather}
  \sum_{n=\epsilon^{-1/2+\delta}}^\infty\frac1n \frac{e^{-\epsilon n}}{1-e^{-\epsilon n}}e^{-bn}(e^{-n X \epsilon}-1)=o(1)
\end{gather}
because $b>0$ and $X\epsilon=o(1)$. Next, for $n\le \epsilon^{-1/2+\delta}$ one has
$|X n \epsilon|=o(1)$. Then,
\begin{gather}
 \sum_{n=1}^{\epsilon^{-1/2+\delta}}\frac1n \frac{e^{-\epsilon n}}{1-e^{-\epsilon n}}e^{-bn}(e^{-n X \epsilon}-1)  \\\sim
 \sum_{n=1}^{\epsilon^{-1/2+\delta}} \frac{1-\epsilon n/2}{\epsilon n^2}e^{-bn}(-X n \epsilon+\frac{(X n \epsilon)^2}2+O(|X n \epsilon|^3)).
\end{gather}
The terms proportional to $X$ and $X^2$  give the r.h.s. of \eqref{eq:R2}.
As for the rest, it is $o(1)$: just recall that $\delta=1/10$ and observe that
\begin{gather}
|X|^3\epsilon^2\ll \epsilon^{2-3/2-3\delta}=o(1).
\end{gather}
\section{Negativity of $\hat R(k)$}

\label{app:Rk}

Recall (cf. \eqref{eq:QAR}) that $\hat R(k)=(1/v)\hat Q(k)$.
We claim that the only stationary points of $\hat Q(k)$ are
$k^{(1)}=(0,0),k^{(2)}=(0,\pi),k^{(3)}=(\pi,0),k^{(4)}=(\pi,\pi)$
(modulo $2\pi$). It is trivial to check that
$\hat Q(k^{(i)})<0$ for $i=2,3,4$ (use that $B,C<D$ and that $x\mapsto f(x):=e^{-x}/(1-e^{-x})$ is decreasing on $\mathbb R^+$) while of course $\hat Q(0,0)=0$. Recall also (cf. \eqref{eq:38}) that the Hessian of $\hat R$ at
$k=0$ is not zero, which implies that $\hat R(k)$ is strictly negative outside $k=0$.

Letting $X=\sin(k_1),Y=\sin(k_2)$, the stationary points of $\hat Q$ must satisfy
\begin{gather}
  \label{eq:stazionari}
 f(D)X=f(C)Y\\
f(D)X=\pm f(B)(X\sqrt{1-Y^2}\pm Y\sqrt{1-X^2}).
\end{gather}
If we exclude the solution $X=Y=0$ (which corresponds to $k=k^{(i)},i=1,\dots,4$), \eqref{eq:stazionari} implies that
\begin{gather}
  \label{eq:Xx}
  X=\pm \frac{\sqrt{\Delta}}{2f(D-C)f(C)f(D)^2}\\
\Delta=-\frac{(1+e^C)(e^C+e^D)(-3e^C+e^{2C}+e^D+e^{C+D})}{(1-e^C)^2(e^C-e^D)^2(1-e^D)^4}
\end{gather}
where we used $B=D-C$.
However, one sees immediately that $\Delta<0$, since $C<D$, so that
\eqref{eq:stazionari} does not give real solutions.

\section{Integral asymptotics}
\subsection{Proof of \eqref{eq:44}}
\label{app:326}

Let $\chi:[-\pi,\pi]^2\to[0,1]$ be a $C^\infty$ cutoff function such that
$\chi(k)=1$ for $|k|<1/2$ and $\chi(k)=0$ for $|k|\ge1$.
We first write the r.h.s. of \eqref{aseconda}
as
\begin{gather}
\label{eq:prima}
\frac v{(2\pi)^2}\int_{\mathbb R^2}dk \chi(k)\frac{e^{ \hat R(k)s}-1}{\hat R(k)}
e^{(t-s)\hat A(k)}e^{-i k y}+j_1(t,s,y)
\end{gather}
where $j_1$, as well as the error terms $j_2,j_3,\ldots$ below, satisfies
\begin{eqnarray}
\label{j1}
j_1=O(1)\quad\text{ and }\quad |j_1|\to0 \quad \text{ if } \max(|y|, t-s)\to\infty.
\end{eqnarray}
To see this,
remark first of all that
\begin{gather}
\label{vm}
 \hat A(k)=\frac 12 \hat R(k)+\frac12(\hat A(k)-\hat A(-k))
\end{gather}
and
\begin{eqnarray}
  \label{eq:49}
 \hat A(k)-\hat A(-k)\in i\mathbb R, \quad \hat R(k)<0, k\ne 0.
\end{eqnarray}
Actually,
\begin{eqnarray}
  \label{eq:51}
  \frac{\hat A(k)-\hat A(-k)}2\stackrel{k\to0}= -i (k, U)+O(|k|^3).
\end{eqnarray}
The function $(1-\chi(k))/\hat R(k)$ is a $C^\infty$ function on $[-\pi,\pi]^2$, so that
\[
 \int_{[-\pi,\pi]^2}dk \frac{e^{(t-s)\hat A(k)}}{\hat
    R(k)}e^{-i k y}(1-\chi(k))
\]
is $o(1)$ when either $t-s\to\infty $ (using $\hat R(k)<0$) or when
$t-s=O(1)$ and $|y|\to \infty$ (the Fourier coefficients of a
$C^\infty$ function on the torus decay faster than any inverse
power). Also,
\begin{eqnarray}
  \label{eq:48}
  \int_{[-\pi,\pi]^2}dk (1-\chi(k))\frac{e^{ \hat R(k)s}}{\hat R(k)}
e^{(t-s)\hat A(k)}e^{-i k y}
\end{eqnarray}
is $o(1)$ when either $s$ or $t-s$ diverge; when $s,t-s=O(1)$ the
integral in \eqref{eq:48} is again $o(1)$ when $|y|\to\infty$ by
decay of Fourier coefficients. Eq. \eqref{eq:prima} follows.

Next, the integral in  \eqref{eq:prima}
gives
\begin{gather}
  \label{j2}
\frac v{(2\pi)^2}\int_{\mathbb R^2}dk \chi(k) \frac{e^{ \hat R(k)s}-1}{\hat W(k)}
e^{(t-s)\hat A(k)}e^{-i k y}+j_2(t,s,y).
\end{gather}
For this, just note that, since $\hat R(k)$ is smooth and even in $k$,
\[
\frac{\hat W(k)-\hat R(k)}{\hat W(k)\hat R(k)}=O(1)
\]
and
\[
\nabla_k \frac{\hat W(k)-\hat R(k)}{\hat W(k)\hat R(k)}=O(1/|k|).
\]
This immediately implies
\begin{gather}
\label{eq:vp}
   \int_{\mathbb R^2}dk \chi(k)e^{\hat R(k)s+\hat A(k)(t-s)}\left(\frac1{\hat
    W(k)}-\frac 1{\hat R(k)}\right)e^{-i k y}=o(1)
\end{gather}
if $\max(t,t-s)\to\infty$ (use dominated convergence, together with
\[
e^{\hat R(k)s+\hat A(k)(t-s)}\to0 \quad \text{for all}\quad k\ne
0\quad \text{if}\quad \max(t,t-s)\to\infty).\] When $\max(t,t-s)=O(1)$ then the l.h.s. of
\eqref{eq:vp} is
$(1/|y|)$
(just do an integration by parts). One bounds
\begin{eqnarray}
  \label{eq:50}
  \int_{\mathbb R^2}dk \chi(k)e^{\hat A(k)(t-s)}\left(\frac1{\hat
    W(k)}-\frac 1{\hat R(k)}\right)e^{-i k y}
\end{eqnarray}
similarly and \eqref{j2} follows.

As a third step, write the integral in \eqref{j2} as
\begin{gather}
\label{j3}
 \frac v{(2\pi)^2}\int_{\mathbb R^2}dk \chi(k) \frac{e^{ \hat W(k)s}-1}{\hat W(k)}
e^{(t-s)\hat A(k)}e^{-i k y}+j_3(t,s,y).
\end{gather}
In fact, assume first that $s\to\infty$.
Then, the contribution to the integral
\begin{gather}
    \int_{\mathbb R^2}dk \chi(k)\frac{e^{\hat R(k)s}-e^{\hat W(k)s}}{\hat
    W(k)}e^{(t-s)\hat A(k)}e^{-i k y}
\end{gather}
from the region $|k|\ge s^{-1/2+\epsilon},\epsilon>0$ is negligible because we have $\hat R(k),\hat W(k)\le -C |k|^2$. The contribution from $|k|\le s^{-1/2+\epsilon}$ is also negligible, this time because
\[
\frac{e^{\hat W(k)s}-e^{\hat R(k)s}}{\hat W(k)}=e^{\hat
  W(k)s}\frac{1-e^{(\hat R(k)-\hat W(k))s}}{\hat W(k)}=O(s |k|)=O(s^{1/2+\epsilon})
\]
(use that $s(\hat R(k)-\hat W(k))=O(s |k|^3)=o(1)$ and $\exp(\hat
W(k)s)\le 1$)
so that the contribution to the integral is
$O(s^{-1/2+3\epsilon})=o(1)$ if $\epsilon$ is small enough.
The proof of \eqref{j3} in the case $s=O(1)$ is simpler and follows
that of \eqref{eq:prima} or \eqref{j2}.

With similar considerations, one rewrites \eqref{j3} as
\begin{gather}
  \label{j4}
 \frac v{(2\pi)^2}\int_{\mathbb R^2}dk \chi(k)\frac{e^{\hat W(k)s}-1}{\hat
    W(k)}e^{(t-s)\hat W(k)/2}e^{-i k y-i(k,U)(t-s)}+j_4(t,s,y)\\=
\frac v{(2\pi)^2} \int_{\mathbb R^2}dk \chi(k)\frac{e^{\hat W(k)s}-1}{\hat
    W(k)}e^{\hat W(k)+(t-s)\hat W(k)/2}e^{-i k y-i(k,U)(t-s)}+j_5(t,s,y)\\=
\frac v{(2\pi)^2} \int_{\mathbb R^2}dk \frac{e^{\hat W(k)s}-1}{\hat
    W(k)}e^{\hat W(k)+(t-s)\hat W(k)/2}e^{-i k y-i(k,U)(t-s)}+j_6(t,s,y).
\end{gather}
In the first step, we used \eqref{vm}--\eqref{eq:51} to replace $\hat
A(k)$ with its second-order Taylor expansion $ \hat W(k)/2-i(k,U)$. In the second, we used that $(\exp(\hat
W(k))-1)/\hat W(k)$ is smooth at $k\sim 0$ in order to multiply by
$\exp(\hat W(k))$. In the third we remove the cutoff function: the
integral with $\chi(\cdot)$ replaced by $1-\chi(\cdot)$ is $o(1)$ either because
$|y|\to\infty$ (integrate by parts w.r.t. $k$) or because
$(t-s)\to\infty$, so that $\exp((t-s)\hat W(k))\to0$ (apply dominated convergence).

\begin{Remark}
The reason why in the second step in \eqref{j4} we multiplied by
$\exp(\hat W(k))$ is
that otherwise we could not remove the
cutoff function $\chi(\cdot)$ since for $t=s$ the integrand would decay
only as $O(|k|^{-2})$ at infinity and the integral would not converge.
  \label{rem:ausi}
\end{Remark}

The integral can now be computed explicitly.
Recall that $\hat W(k)=(k,\hat W k)$ with $\hat W$ a strictly
negative definite symmetric matrix.
From definition \eqref{eq:Vv} we have
\begin{eqnarray}
  \label{eq:18}
 | \det(V)|=\frac1{w}.
\end{eqnarray}
Changing variables as $k=V^T w$ and putting $H=Vy+(t-s)V U$, the integral in the r.h.s. of \eqref{j4} becomes
\begin{gather}
\label{eq:porta}
 -|\det(V)|\frac v{(2\pi)^2}\int_{\mathbb R^2}e^{-i (w, H)}\frac{e^{-s|w|^2}-1}{|w|^2}e^{-|w|^2(1+(t-s)/2)}dw\\
= -|\det(V)|\frac v{(2\pi)^2}\int_1^\infty da\int_{\mathbb R^2}e^{-i (w,H)}(e^{-s|w|^2}-1)e^{-(a+(t-s)/2)|w|^2}dw\\
=- \pi|\det(V)|\frac v{(2\pi)^2}\int_1^\infty\left[\frac{e^{-|H|^2/(4(a+(t+s)/2))}}{a+(t+s)/2}-\frac{e^{-|H|^2/(4(a+(t-s)/2))}}{a+(t-s)/2}\right]da\\
= |\det(V)|\frac
v{4\pi}\int_{1+(t-s)/2}^{1+(t+s)/2}\frac{e^{-|H|^2/(4a)}}a da
\end{gather}
Together with \eqref{eq:18}, Eq. \eqref{eq:44} then follows.

\subsection{Proof of \eqref{4punti}}
\label{app:322}
Similar (actually simpler) arguments as those leading to \eqref{j4}
show that the integral in \eqref{eq:moreg} equals
\begin{gather}
  \label{eq:39bis}
-\frac v{(2\pi)^2}\int_{\mathbb R^2}d k\frac{(e^{i
  k y_1}-e^{i k y_2})(e^{-i k y_3}-e^{-i k y_4})}{\hat W(k)}e^{\hat
  W(k)}+R_{y_1,y_2,y_3,y_4}
\end{gather}
with $R$ satisfying \eqref{eq:R4}.
The factor $\exp(\hat W(k))$ appears for the same reasons as in
Remark \ref{rem:ausi}.
With the same change of coordinates $k=V^T w$ as before, the integral
in
\eqref{eq:39bis} becomes
\begin{eqnarray}
  \label{eq:41}
\frac v{(2\pi)^2w} \int_{\mathbb
  R^2}\frac{(e^{i (w,Y_1) }-e^{i (w,Y_2)})(e^{-i (w,Y_3)}-e^{-i (w,Y_4)})}{|w|^2}e^{-|w|^2}d w
\end{eqnarray}
with $Y_i=V y_i$.
This equals
\begin{gather}
  \label{eq:42}
\frac v{4\pi w} 
\int_1^\infty
da\frac{e^{-|Y_1-Y_3|^2/(4a)}-e^{-|Y_1-Y_4|^2/(4a)}+e^{-|Y_2-Y_4|^2/(4a)} -e^{-|Y_2-Y_3|^2/(4a)}}a\\
=\frac v{(2\pi)^2w} \left[
-\Gamma\left(0,\frac{|Y_1-Y_3|^2}4\right)+\Gamma\left(0,\frac{|Y_1-Y_4|^2}4\right) \right.\\\left.-\Gamma\left(0,\frac{|Y_2-Y_4|^2}4\right) +\Gamma\left(0,\frac{|Y_2-Y_3|^2}4\right)+2\log\frac{|Y_1-Y_4||Y_2-Y_3|}{|Y_1-Y_3||Y_2-Y_4|}
\right],
\end{gather}
where $\Gamma(0,x):=\int_x^\infty
e^{-t}/t \;dt$ is the incomplete Gamma function.
Using  the exponential decay of $\Gamma(0,x)$ as $x\to\infty$, we get
immediately \eqref{4punti} (the ``$+1$''s in
\eqref{4punti} takes care of the case where $Y_1=Y_3$ and/or
$Y_2=Y_4$: in fact, $-\Gamma(0,x^2/4)+\log (1/x^2)$ has a finite limit
as $x\to 0$).

\end{document}